\documentclass[a4paper,12pt]{amsart}

\usepackage{amsmath,amssymb,amscd}

\newcommand{\hbC}{\hat{\mathbb{C}}}
\newcommand{\bN}{\mathbb{N}}
\newcommand{\bC}{\mathbb{C}}
\newcommand{\bR}{\mathbb{R}}
\newcommand{\bD}{\mathbb{D}}

\newcommand{\TS}{\mathit{TS}}
\newcommand{\AT}{\mathit{AT}}
\newcommand{\PB}{\mathit{PB}}
\newcommand{\CM}{\mathit{CM}}
\newcommand{\UH}{\mathit{UH}}

\newcommand{\Int}{\operatorname{int}}
\newcommand{\diam}{\operatorname{diam}}
\newcommand{\rd}{\mathit{d}}

\theoremstyle{plain}
\newtheorem{theorem}[equation]{Theorem}
\newtheorem{lemma}[equation]{Lemma}
\newtheorem{corollary}[equation]{Corollary}
\newtheorem{mainth}{Theorem}
\newtheorem{maincoro}{Corollary}

\theoremstyle{definition}
\newtheorem{definition}[equation]{Definition}
\newtheorem{notation}[equation]{Notation}
\newtheorem*{ac}{Acknowledgement}
\theoremstyle{remark}
\newtheorem{remark}[equation]{Remark}
\newtheorem*{remarks}{Remarks}
\newtheorem*{supplement}{Supplement by D. Drasin}

\numberwithin{equation}{section}

\renewcommand{\bold}[1]{\smallskip \noindent {\bf \boldmath #1 }\nopagebreak[4]}

\begin{document}

\title[singularities and unhyperbolicity]{Singularities of Schr\"oder maps and unhyperbolicity of rational functions}

\author[David Drasin]{David Drasin}
\address{Department of Mathematics, Purdue University, West Lafayette, IN 47907, USA}
\email{drasin@math.purdue.edu}

\author[Y\^usuke Okuyama]{Y\^usuke Okuyama}
\address{Department of Comprehensive Sciences,
Graduate School of Science and Technology,
Kyoto Institute of Technology,
Kyoto 606-8585, JAPAN}
\email{okuyama@kit.ac.jp}



\dedicatory{Dedicated to Professor Walter K. Hayman on his eightieth birthday}

\date{July 30, 2007}

\begin{abstract}
 We study transcendental singularities of
 a Schr\"oder map arising from a rational function $f$,
 using results from complex dynamics and Nevanlinna theory.
 These maps are transcendental meromorphic functions
 of finite order in the complex plane.
 We show that their transcendental singularities
 lie over the set where $f$ is not semihyperbolic (unhyperbolic). In addition,
 if they are direct, then they lie over only attracting periodic points of $f$, 
 and moreover, if $f$ is a polynomial, 
 then both direct and indirect singularities lie over
 attracting, parabolic and Cremer periodic points of $f$.
 We also obtain concrete examples of
 both kinds of transcendental singularities of Schr\"oder maps
 as well as a new proof of the Pommerenke-Levin-Yoccoz inequality 
 and a new formulation of the Fatou conjecture.
\end{abstract}

\maketitle


\section{Introduction}\label{sec:intro}
Let $f$ be a rational function on $\hbC$
of degree $d=\deg f\ge 2$, i.e., 
the critical set $C(f):=\{f'(c)=0\}\neq\emptyset$.
Denote its $k$-th iterate ($k\in\bN\cup\{0\}$) by $f^k$.
For details of complex dynamics, see, for example, \cite{Milnor3rd},
\cite{MorosawaTaniguchi}, \cite{Steinmetz93}. 
For every {\itshape repelling} periodic point $z_0$ of $f$
of period $p$, there exists a unique meromorphic map $h$ on $\bC$,
which is called the {\itshape Schr\"oder map of $f$ at $z_0$},
such that $h(0)=z_0$, $h'(0)=1$ and
\begin{align}
 f^{p}\circ h=h\circ\lambda\label{eq:Schroder}
\end{align}
on $\bC$. Here the multiplier $\lambda:=(f^p)'(z_0)$ ($|\lambda|>1$)
also denotes multiplication by $\lambda$ on $\bC$.
Using complex dynamics and Nevanlinna theory, 
we study the relationship between
{\itshape singularities} of Schr\"oder maps $h$
and the {\itshape unhyperbolicity} of $f$. 
Following Carleson-Jones-Yoccoz \cite{CJY94},
we say that $f$ is {\itshape not semihyperbolic} or, more conveniently,
{\itshape unhyperbolic} at $a\in\hbC$ if for every open neighborhood $U$ of $a$,
\begin{gather}
 \sup_{k\in\bN}\max_{V^{-k}}\deg(f^k:V^{-k}\to U)=\infty, \label{eq:unhyp}
\end{gather}
where $V^{-k}$ ranges over all components of $f^{-k}(U)$.
We denote by $\UH(f)$ (unhyperbolic) the set of all such $a\in\hbC$.

\begin{notation}
 $U_r(a)$ is the spherical open disk centered at  $a\in\hbC$ and of radius $r>0$.
 Let $F(f)$ and $J(f)$ be the Fatou and Julia sets of $f$, respectively,
 and let $\AT(f),\PB(f)$ and $\CM(f)$ be 
 the attracting, parabolic and Cremer periodic points of $f$, respectively.
\end{notation}

If $g$ is transcendental meromorphic on $\bC$,
we can consider more general singularities than its critical set $C(g)$:
let $\mathfrak{N}$ be the set of {\itshape decreasing} families
$\mathcal{A}=\{A_r\}_{r>0}\subset 2^{\bC}$, so that $A_s\subset A_r$ if $s<r$. 
Let $\TS(g)\subset\mathfrak{N}$ be the set of
$\mathcal{A}\in\mathfrak{N}$ such that 
there exists (the unique) $a=a_{\mathcal{A}}\in\hbC$ such that for every $r>0$,
$A_r$ is a component of $g^{-1}(U_r(a))$ and in addition
that $\bigcap_{r>0}A_r=\emptyset$ (cf.\ \cite{BE95}). 
Each $\mathcal{A}\in\TS(g)$ is called a {\itshape transcendental singularity of $g$},
and we extend $g$ to the map from $\bC\cup\TS(g)$ to $\hbC$ by 
setting $g(\mathcal{A}):=a$ for $\mathcal{A}\in\TS(g)$.
Following terminology due to Iversen \cite{Iversen14}, 
$\mathcal{A}$ is said to be {\itshape direct}
if the point $g(\mathcal{A})$ is not contained in $g(A_r)$ for some $r>0$,
and {\itshape indirect} otherwise.

For a sequence $(z_k)\subset\bC$, we say that 
$z_k\to\mathcal{A}$ as $k\to\infty$ if for each $r>0$,
$z_k\in A_r$ for all large $k\in\bN$.
Similarly, for an arc $\gamma:(-\infty,\infty)\to\bC$, 
we say that $\gamma(t)\to\mathcal{A}$ as $t\to\infty$ 
if for each $r>0$, $\gamma(t)\in A_r$ for all large $t>0$.
We call $\gamma$ an {\itshape asymptotic arc} of $g$
if $\lim_{t\to\infty}\gamma(t)=\infty$ and $\lim_{t\to\infty}g(\gamma(t))\in\hbC$ exists.
For every asymptotic arc $\gamma$ of $g$,
there is the (unique) $\mathcal{A}\in\TS(g)$ {\itshape associated with} $\gamma$,
so that $\gamma(t)\to\mathcal{A}$ as $t\to\infty$.
Conversely, for every $\mathcal{A}\in\TS(g)$,
there is an asymptotic arc $\gamma$ of $g$ to  which $\mathcal{A}$ is associated;
indeed many such arcs.

\begin{remarks}
 This definition of $\mathcal{A}$ slightly modifies the classical
 one (eg.\ in \cite[\S XI]{Nevan70}), where $g(\mathcal{A})$ 
 is called a transcendental singularity {\itshape of $g^{-1}$}.
 In the study of entire-meromorphic maps, the term
 {\itshape asymptotic curve} is more common than asymptotic arc,
 but we prefer the latter since it seems more in keeping with usage in dynamics.
\end{remarks}

For $\mathcal{A}=\{A_r\},\mathcal{B}=\{B_r\}\in\mathfrak{N}$,
we say $\mathcal{A}\sim\mathcal{B}$ if
\begin{itemize}
 \item for every $r>0$, there exists $s>0$ such that $A_s\subset B_r$,
       \label{item:pushtail}
 \item for every $r>0$, there exists $s>0$ such that $B_s\subset A_r$.
       \label{item:directindirect}
\end{itemize}
This defines an equivalence relation on $\mathfrak{N}$.

When $g$ is a Schr\"oder map $h$ as (\ref{eq:Schroder}),
we call the map $\Lambda=\Lambda_h$ below
the {\itshape natural} extension of the multiplication action of $\lambda$ on $\bC$
since from (\ref{eq:extension}), we have $h\circ\Lambda=f^p\circ h$ on $\TS(h)$.

\begin{theorem}\label{th:extension}
 Let $f$ and $h$ be as above. Then
 there exists a map $\Lambda=\Lambda_h:\TS(h)\to\TS(h)$
 such that for each $\mathcal{A}=\{A_r\}\in\TS(h)$,
 \begin{gather}
  \Lambda\mathcal{A}\sim\{\lambda A_r\}_{r>0}.\label{eq:extension}
 \end{gather}
 The map $\Lambda$ is bijective and
 preserves the direct or indirect character of $\mathcal{A}\in\TS(h)$, i.e.,
 $\mathcal{A}$ is direct if and only if $\Lambda\mathcal{A}$ is direct.
\end{theorem}

\begin{definition}
 An $\mathcal{A}\in\TS(h)$ is {\itshape periodic} if it is periodic under $\Lambda_h$. 
\end{definition}
Theorem \ref{th:extension} is shown by
a careful chase of the functional equation (\ref{eq:Schroder}).
For reader's convenience, we include the proof in \S \ref{sec:extension}.

In some ways, this paper may be viewed as a continuation of \cite{DOproximity},
which studies the growth with $k$ of the proximity function $m(a,f^k)$
as $a$ varies in $\hbC$. 
Thus as in \cite{DOproximity}, we consider the {\itshape omega-limit set} 
\begin{gather*}
 \omega_f(c):=\{z\in\hbC;\exists k_j\to\infty\text{ such that }\lim_{j\to\infty}f^{k_j}(c)=z\}
\end{gather*}
 for each $c\in\hbC$, and define the {\itshape Ma\~n\'e set} of $f$ as
\begin{align*}
 M(f):=\bigcup_{\substack{c\in C(f)\cap J(f)\\ \text{such that } c\in\omega_f(c)}}\omega_f(c).
\end{align*}
We will recall in \S \ref{sec:nevan} that $\CM(f)\subset M(f)$
and that $\AT(f)\cup\PB(f)\cup M(f)$ coincides with $\UH(f)$.

One of our principal results is:
\begin{mainth}\label{th:asym}
Let $h$ be a Schr\"oder map of the rational function $f$. Then
\begin{gather}
 h(\TS(h))\subset\AT(f)\cup\PB(f)\cup M(f),\label{eq:unhyperbolic}\\
 h(\{\mathcal{A}\in\TS(h);\text{periodic}\})
 \subset\AT(f)\cup\PB(f)\cup\CM(f),\label{eq:periodic}\\
 h(\{\mathcal{A}\in\TS(h);\text{direct}\})\subset\AT(f).\label{eq:directasym}
\end{gather}
\end{mainth}

In general, the inclusion (\ref{eq:directasym}) is proper. 
As an example, we have: 
\begin{mainth}\label{th:attpara}
 Let $h$ be a Schr\"oder map of the rational function $f$
 at a repelling fixed point $z_0$ of $f$ of multiplier $\lambda$,
 let $D$ an immediate basin of $a\in\AT(f)\cup\PB(f)$,
 and suppose that a component $W$ of $h^{-1}(D)$ is periodic $($under $\lambda)$
 in that $\lambda^N W=W$ for some $N\in\bN$ $($so $f^N(D)=D)$.
 Then for every $w_0\in W$, there is an asymptotic arc 
 $\gamma:(-\infty,\infty)\to W$ of $h$ with $\gamma(0)=w_0$ and
 $\lim_{t\to\infty}h(\gamma(t))=a$ such that for every $t\in(-\infty,\infty)$,
 \begin{gather}
  \gamma(t+1)=\lambda^N\gamma(t).\label{eq:translation}
 \end{gather} 
 Suppose that $a\in\AT(f)$.
 If $f^{-N}(a)\cap D\neq\{a\}$, then $W\cap h^{-1}(a)\neq\emptyset$. 
 If $\gamma(0)=w_0\in W\cap h^{-1}(a)$,
 then $\mathcal{A}\in\TS(h)$ associated with this $\gamma$
 is indirect.
 If there exists an indirect $\mathcal{A}=\{A_r\}\in\TS(h)$ 
 with $h(\mathcal{A})=a$ and $A_r\subset W$ for some $r>0$,
 then $D\cap C(f^N)\neq\{a\}$.
\end{mainth}

Replacing $f$ by $f^n$ for an appropriate $n\in\bN$,
we may apply Theorem \ref{th:attpara} to every immediate basin $D$
of $a\in\AT(f)\cup\PB(f)$, 
by a theorem of Przytycki-Zdunik \cite[Theorem A]{PZ94}
(see also Pommerenke \cite[\S 2]{Pommerenke86} when
$D$ is simply connected):
{\itshape $\partial D$ contains a dense subset of
repelling periodic points $z_0$ of $f$ 
accessible from $D$ along an arc $\Gamma=\Gamma_{z_0}:[0,1]\to\overline{D}$
such that $\Gamma(0)=z_0$, $\Gamma((0,1])\subset D$ and 
$\Gamma\subset f^n(\Gamma)$ for some $n\in\bN$.}
We note that $z_0$ is {\itshape fixed by} $f^n$, and put $\lambda:=(f^n)'(z_0)$.
For all small $s>0$, the Schr\"oder map $h$ of $f^n$ at $z_0$ is univalent on $\{|w|<s\}$,
and we can assume that $\Gamma\subset(f^n)(\Gamma)\subset h(\{|w|<s\})$.
Choose $W$ as $(h|\{|w|<s\})^{-1}(\Gamma)\subset W$.
Then we have $W=\lambda W$.

\begin{maincoro}\label{th:attparaconverse}
 For every $a\in\AT(f)\cup\PB(f)$ of the rational function $f$,
 there exists a Schr\"oder map $h$ of $f$ with $a\in h(\TS(h))$.
\end{maincoro}

We apply these results to two concrete dynamical issues.

\bold{Pommerenke-Levin-Yoccoz inequality.}
Suppose that the rational function $f$ is a polynomial and
that $p=1$, i.e., $f(z_0)=z_0$. Then $h$ is an entire function
of order $\rho=(\log d)/\log|\lambda|$ (see \S \ref{sec:nevan}).
Let $D_{\infty}$ be the immediate basin of $\infty$,
and let $q_{\infty}=q_{\infty}(h)$ be the number of components of $h^{-1}(D_{\infty})$.
Eremenko and Levin \cite{EremenkoLevin89} proved that $q_{\infty}\le 2\rho\vee 1$,
that every component of $h^{-1}(D_{\infty})$ is periodic under $\lambda$ and
(\cite[p.\ 1260]{EremenkoLevin89}) that
if $J_0$ is the component of $J(f)$ with $z_0\in J_0$ and
\begin{gather}
 J_0\neq\{z_0\}\tag{EL}\label{eq:ELcondition}
\end{gather}
(eg., if there are at least two components of $h^{-1}(D_{\infty})$), then
\begin{gather}
 q_{\infty}\le 2\rho.\label{eq:EL}
\end{gather}

A {\itshape spiral version} of Denjoy's conjecture (see Theorem \ref{th:spiral} below),
which was considered by Ahlfors \cite{Ahlfors30}
and proved unambiguously by Hayman/Jenkins
(cf.\ \cite[Theorem 8.21]{Hayman89}, \cite{Jenkins87})
establishes a refinement of (\ref{eq:EL}):
condition (\ref{eq:ELcondition}) implies that
for every asymptotic arc $\gamma:(-\infty,\infty)\to\bC$ of $h$
with $\lim_{t\to\infty}h(\gamma(t))=\infty$,
\begin{gather}
 q_{\infty}\cdot\left(1+\limsup_{t\to\infty}\left(\frac{\arg\gamma(t)}{\log|\gamma(t)|}\right)^2\right)\le 2\rho.\label{eq:general}
\end{gather}

As a special case, (\ref{eq:general}) has a {\itshape dynamical} implication:
for each component $W$ of $h^{-1}(D_{\infty})$,
let $q_W$ be the least $N\in\bN$ such that $\lambda^NW=W$, and
let $\gamma_W:(-\infty,\infty)\to W$ be an asymptotic arc of $h$ 
obtained by Theorem \ref{th:attpara}.
Then by (\ref{eq:translation}), for every $k\in\bN$, we have
\begin{gather}
 \gamma_W(k)=\lambda^{k\cdot q_W}\gamma_W(0).\label{eq:orbits}
\end{gather}
If (\ref{eq:ELcondition}) holds, then
we can define a single-valued branch $\arg_W(\cdot)$ of $\arg(\cdot)$ on $W$,
and there exists a (unique) $p_W\in\{0,1,\ldots,q_W-1\}$
such that for some branch of $\arg\lambda$,
\begin{gather}
 \arg_W(\lambda^{q_W}w)-\arg_W(w)=q_W\arg\lambda-2\pi p_W\label{eq:argument}
\end{gather}
for every $w\in W$. We can also show that both $q=q_W$ and $p=p_W$ are independent of $W$
(see the discussion of (\ref{eq:cyclic}) below). Therefore,

\begin{maincoro}\label{th:PLY}
Let $h$ be a Schr\"oder map of a polynomial $f$ of degree $d\ge 2$
at a repelling fixed point $z_0$ of $f$ having multiplier $\lambda$.
Assume $($\ref{eq:ELcondition}$)$, and let $q_{\infty},p,q$ be as above. Then 
there is a branch of $\arg\lambda$ so that
\begin{gather}
 q_{\infty}\cdot\left(1+\left(\frac{\arg\lambda-2\pi p/q}{\log|\lambda|}\right)^2\right)\le 2\rho=2\frac{\log d}{\log|\lambda|}.\label{eq:ELPLY}
\end{gather}
\end{maincoro}

Inequality (\ref{eq:ELPLY}) was shown
by Pommerenke \cite{Pommerenke86} and Levin \cite{Levin91}
in somewhat weaker form,
by Yoccoz (unpublished) in an equivalent form to (\ref{eq:ELPLY})
under the assumption $J_0=J(f)$,
and by Jin \cite{Jin03} under (\ref{eq:ELcondition}).

\bold{Fatou conjecture.}
We consider the {\itshape unicritical} polynomial family
\begin{gather*}
 \{f_c(z)=z^d+c;c\in\bC\}\cong\bC, 
\end{gather*}
and note that $C(f_c)\cap\bC=\{0\}$ while $\infty\in\AT(f_c)$.
The {\itshape Mandelbrot set} and
its {\itshape hyperbolicity locus} are defined as
\begin{align*}
 &\mathcal{C}:=\{c\in\bC;\lim_{k\to\infty}|f_c^k(0)|\neq\infty\},
 &\mathcal{H}:=\{c\in\bC;\AT(f_c)\cap\bC\neq\emptyset\}
\end{align*}
respectively. It is known that $\mathcal{H}$ is an open and closed subset of 
$\Int\mathcal{C}$ (\cite[Theorem 4.4]{McMullen:renorm}).

We say that a covering selfmap $g$ of $\bC$, 
which is possibly ramified and not surjective, 
covers a point $a\in\bC$ {\itshape completely}
if there exists $r>0$ such that $g^{-1}(U_r(a))$ has no unbounded component;
$g$ itself is {\itshape complete} if it covers all $a\in\bC$ completely
(cf. \cite[I.\ 21A]{AhlforsSario}).

\begin{maincoro}\label{th:hypdense}
 Let $c\in\Int\mathcal{C}$. Then $c\not\in\mathcal{H}$ if and only if
 every Schr\"oder map of $f_c$ is a complete covering selfmap of $\bC$.
\end{maincoro}

We remark that it has been expected for a long time that
\begin{gather*}
 \Int\mathcal{C}=\mathcal{H};\label{eq:complexFatou}
\end{gather*} 
this is known as a {\itshape Fatou conjecture}
(cf.\ \cite[p.\ 58]{McMullen:renorm}).
Perhaps our characterization of $\mathcal{H}$
might be helpful in understanding this conjecture.

\section{Dynamical and Nevanlinna-theoretic results}\label{sec:nevan}
Let $f$ be a rational function on $\hbC$ of degree $d\ge 2$.

\bold{Ma\~n\'e's theorem and Siegel compacta.}
Consider the set $\UH(f)$ in (\ref{eq:unhyp}).
By a standard argument (cf. \cite[\S 2]{DOproximity}), we have
$\UH(f)\cap F(f)=\AT(f)$ and $\UH(f)\cap J(f)\supset \PB(f)\cup M(f)$,
and Ma\~n\'e's theorem below sharpens the second containment to equality, so that
\begin{gather*}
 \UH(f)=\AT(f)\cup\PB(f)\cup M(f).
\end{gather*}

\begin{theorem}[{\cite[Theorem II]{Mane93}}]\label{th:recurrent}
 For every $a\in J(f)\setminus(\PB(f)\cup M(f))$
 and every $\epsilon>0$, there exists an open neighborhood $U$ of $a$ such that
 \begin{gather}
  \begin{aligned}
   &\sup_{k\in\bN}\max_{V^{-k}}\diam V^{-k}\le\epsilon,
   &\sup_{k\in\bN}\max_{V^{-k}}\deg(f^k:V^{-k}\to U)\le d^{2d-2},
  \end{aligned}\notag\\
  \begin{gathered}
   \lim_{k\to\infty}\max_{V^{-k}}\diam V^{-k}=0,\label{eq:schrink}
  \end{gathered}
 \end{gather}
where $V^{-k}$ ranges over all components of $f^{-k}(U)$.
\end{theorem}

We also use P\'erez-Marco's theorem on indifferent fixed points. 

\begin{theorem}[{\cite[Theorem 1]{PM97}}]\label{th:hedgehog}
 Let $\phi$ be an analytic germ at an indifferent fixed point $x\in\hbC$,
 which is univalent on an open set compactly containing
 a Jordan neighborhood $U\subset\hbC$ of $x$.
 Then there exists a continuum $K\subset\overline{U}$,
 which is called a Siegel compactum associated to $(\phi,U,x)$,
 such that $x\in K=\phi(K)\not\subset U$ and $\hbC\setminus K$ is connected.
\end{theorem}

\bold{Meromorphic maps of finite order.}
Let $g$ be a meromorphic map on $\bC$. The {\itshape order} of $g$ is
$\rho=\rho_g:=\limsup_{r\to\infty}(\log T(r,g))/(\log r)\in[0,+\infty]$
(cf. \cite[p.\ 215]{Nevan70}). 
When $\rho_g<\infty$, as occurs here (compare (\ref{eq:Valiron}) below),
$g$ is subject to two fundamental controls.

\begin{theorem}[Denjoy-Carleman-Ahlfors, cf. {\cite[p.\ 303, p. 307]{Nevan70}}]\label{th:DCA}
\begin{gather}
 \#\{\mathcal{A}\in\TS(g);\mathcal{A}\text{ is direct}\}\le 2\rho\vee 1.\label{eq:DCA}
\end{gather}
Moreover, if $g$ is entire, then
\begin{gather}
 \#\{\mathcal{A}\in\TS(g);g(\mathcal{A})\in\bC\}\le 2\rho.\label{eq:DCAentire}
\end{gather}
\end{theorem}

\begin{theorem}[Bergweiler-Eremenko {\cite[Theorem 1']{BE95}}]\label{th:indirect}
 If $\mathcal{A}=\{A_r\}\in\TS(g)$ is indirect, then
 there exists $(c_k)\subset C(g)\setminus g^{-1}(g(\mathcal{A}))$
 such that $c_k\to\mathcal{A}$ as $k\to\infty$.
 In particular, $g(\mathcal{A})$ is a derived point of $g(C(g))$. 
\end{theorem}

We record one consequence of Theorem \ref{th:DCA}.

\begin{lemma}\label{th:directunique}
 Suppose that $\rho_g<\infty$.
 If $\mathcal{A}=\{A_r\}\in\TS(g)$ is direct and $r>0$ is small enough, 
 then for every $t\in(0,r)$,
 $A_t$ is the only component of $g^{-1}(U_t(g(\mathcal{A})))$ contained in $A_r$.
\end{lemma}

\begin{proof}
 Otherwise, there exists $(r_j)_{j\in\bN}\subset\bR_{>0}$ 
 decreasing to $0$ such that for each $j\in\bN$,
 there is a component $B_{r_{j+1}}$ of $g^{-1}(U_{r_{j+1}}(g(\mathcal{A})))$
 other than $A_{r_{j+1}}$ and contained in $A_{r_j}$.
 Since $\mathcal{A}$ is direct, 
 we may assume that $h(A_{r_1})\not\ni g(\mathcal{A})$.
 Then for every $j\ge 2$, 
 there exists $\mathcal{B}^j=\{B^j_r\}\in\TS(g)$ such that
 $g(\mathcal{B}^j)=g(\mathcal{A})$ and 
 $B^j_{r_j}=B_{r_j}\subset A_{r_{j-1}}\setminus A_{r_j}$,
 so that all $\mathcal{B}^j$ are not only direct but also mutually distinct.
 This contradicts (\ref{eq:DCA}).
\end{proof}

\bold{Schr\"oder maps.}
When $g$ is a Schr\"oder map $h$ of $f$ as in \S $\ref{sec:intro}$, 
Valiron calculated that
\begin{gather}
 \rho_h=\frac{\log d^{p}}{\log|\lambda|}(<\infty)\label{eq:Valiron} 
\end{gather}
(cf. \cite[p.\ 160]{Valiron54}), 
so those results may be applied to $h$.

The next theorem seems well known. We sketch a proof for completeness.

\begin{theorem}\label{th:critvalue} 
For every Schr\"oder map $h$ of $f$ with $p=1$,
\begin{gather}
 E(f):=\{a\in\hbC;f^{-2}(a)=\{a\}\}=\{a\in\hbC;h^{-1}(a)=\emptyset\},\label{eq:exceptional}\\
 h(C(h))=\bigcup_{k\in\bN}f^k(C(f))\setminus E(f).\label{eq:critvalue}
\end{gather}
\end{theorem}

\begin{proof}
 We recall (cf. \cite[Lemma 4.9]{Milnor3rd} or \cite[Theorem 2.3.3]{MorosawaTaniguchi}) that $E(f)\subset C(f)$ and that
\begin{gather}
  E(f)=\{a\in\hbC;\#\bigcup_{j\in\bN}f^{-j}(a)<\infty\}.\label{eq:backwardfinite}
 \end{gather}
 Let $E_P(h)^*$ be the last term in (\ref{eq:exceptional}).
 Suppose that $a\in\hbC\setminus E(f)$. 
 Then by (\ref{eq:backwardfinite}), $\#\bigcup_{j\in\bN}f^{-j}(a)=\infty(\ge 3)$, 
 so that by Picard's theorem,
 $h(\bC)\cap\bigcup_{j\in\bN}f^{-j}(a)\neq\emptyset$.
 This with (\ref{eq:Schroder}) implies that $h^{-1}(a)\neq\emptyset$,
 i.e, $a\in\hbC\setminus E_P(h)^*$. Conversely,
 suppose that $a\in\hbC\setminus E_P(h)^*$.
 Note that for every $j\in\bN$, by repeated use of (\ref{eq:Schroder}),
 \begin{gather}
  h^{-1}(f^{-j}(a))=\lambda^{-j}(h^{-1}(a)).\label{eq:pullbackpreimage}
 \end{gather} 
 If $a\in E(f)$, then by the first equality in (\ref{eq:exceptional}),
 we would have $h^{-1}(a)=h^{-1}(f^{-2j}(a))=\lambda^{-2j}(h^{-1}(a))$
 for every $j\in\bN$. Hence since $|\lambda|>1$ and $h$ is continuous at $0$,
 we would have $0\in h^{-1}(a)$, and $z_0=h(0)=a\in(E(f)\subset)C(f)$.
 This contradicts that $f'(z_0)=\lambda\neq 0$.

 We have shown (\ref{eq:exceptional}).
 We can show (\ref{eq:critvalue}) by repeated use of (\ref{eq:Schroder}),
 the chain rule and the fact that $h'(w)\neq 0$ if $|w|$ is small enough.
\end{proof}

\begin{remark}
We note another description of $E(f)$: 
 \begin{gather*}
  E(f)=E_P(h):=\{a\in\hbC;\# h^{-1}(a)<\infty\}.
 \end{gather*} 
 By (\ref{eq:exceptional}), $E(f)=E_P(h)^*\subset E_P(h)$.
 Conversely, suppose that $a\in E_P(h)$. Then by (\ref{eq:pullbackpreimage}), 
 we have for every $j\in\bN$, $\#h^{-1}(f^{-j}(a))=\# h^{-1}(a)<\infty$.
 Hence $\bigcup_{j\in\bN}f^{-j}(a)\subset E_P(h)$, and
 $\# E_P(h)\le 2$ by Picard's theorem.
 Thus by (\ref{eq:backwardfinite}), we have $a\in E(f)$.

 Since $E(f^k)=E(f)$ for every $k\in\bN$,
 we have $E(f)=E_P(h)^*=E_P(h)$ even when $p\ge 2$.
\end{remark}

\begin{corollary}\label{th:postcrit}
Let $h$ be a Schr\"oder map of $f$. Then
\begin{enumerate}
 \item if $f$ is a polynomial, then $h$ is entire$;$
       \label{item:entire}
 \item every direct $\mathcal{A}\in\TS(h)$ is periodic.
       If $h$ is also entire, then every $\mathcal{A}\in\TS(h)$ is periodic$;$  
       \label{item:periodic}
 \item for every indirect $\mathcal{A}\in\TS(h)$,
       $h(\mathcal{A})$ is a derived point of the critical orbit
       $\bigcup_{k\in\bN}f^k(C(f))$ of $f$.
       \label{item:postcrit}
\end{enumerate}
\end{corollary}

\begin{proof}
 The assertion (\ref{item:entire}) follows from (\ref{eq:exceptional}) 
 with $a=\infty\in E(f)$,
 (\ref{item:periodic}) from Theorems \ref{th:extension} and \ref{th:DCA},
 and (\ref{item:postcrit}) from (\ref{eq:critvalue}) and Theorem \ref{th:indirect}.
\end{proof}

\section{Proof of Theorem \ref{th:extension}}\label{sec:extension}
For a meromorphic map $g$ on $\bC$, the following is straightforward:
for $\mathcal{A}=\{A_r\},\mathcal{B}=\{B_r\}\in\TS(g)(\subset\mathfrak{N})$,
\begin{gather}
 \mathcal{A}=\mathcal{B}\Leftrightarrow\mathcal{A}\sim\mathcal{B}.\label{eq:equivalence}
\end{gather}

Replacing $f^p$ by $f$ if necessary,
we assume that $p=1$. Put $\tilde{h}:=h\circ\lambda^{-1}$.
For every $\mathcal{A}=\{A_r\}\in\TS(h)$, 
we have $\{\lambda A_r\}=:\tilde{\mathcal{A}}=\{\tilde{A}_r\}\in\TS(\tilde{h})$.
We show that there exists the (unique)
$\mathcal{B}=\{B_r\}\in\TS(f\circ\tilde{h})(=\TS(h)$ since $f\circ\tilde{h}=h$)
with $\mathcal{B}\sim\tilde{\mathcal{A}}$:
put $a:=\tilde{h}(\tilde{\mathcal{A}})(=h(\mathcal{A}))$ and $b:=f(a)$.
For every $r>0$, there exists $s>0$ such that $U_s(a)\subset f^{-1}(U_r(b))$,
and let $B_r$ be the component of
$\tilde{h}^{-1}(f^{-1}(U_r(b)))$ with $\tilde{A_s}\subset B_r$. Then we have
$\mathcal{B}:=\{B_r\}\in\mathfrak{N}$.
Conversely, for every $r>0$, if $s>0$ is small enough, then
$(f(\tilde{h}(B_s))\subset)U_s(b)\subset f(U_r(a))$,
so that $\tilde{h}(B_s)\subset f^{-1}(f(U_r(a)))$.
Since $\tilde{A}_t\subset B_s\cap\tilde{A}_r$ for some $t>0$,
we have $\emptyset\neq\tilde{h}(B_s)\cap U_r(a)$.
Hence if $r>0$ is so small that $f:U_r(a)\to f(U_r(a))$ is proper, 
then $\tilde{h}(B_s)\subset U_r(a)$,
and so $(\emptyset\neq\tilde{A}_t\subset)B_s\subset\tilde{A}_r$.
Hence we have $\mathcal{B}\sim\tilde{\mathcal{A}}$,
and $\mathcal{B}\in\TS(f\circ\tilde{h})$.

We define $\Lambda:\TS(h)\to\TS(h)$ by $\Lambda\mathcal{A}:=\mathcal{B}$,
which is injective by (\ref{eq:equivalence}).
Moreover, if $r>0$ is small enough, then $U_r(a)\cap f^{-1}(b)=\{a\}$,
which implies that $\mathcal{B}=\Lambda\mathcal{A}$ is direct
if and only if so is $\mathcal{A}$.

We show the surjectivity of $\Lambda:\TS(h)\to\TS(f\circ\tilde{h})=\TS(h)$:
let $\mathcal{B}=\{B_r\}\in\TS(f\circ\tilde{h})$ and
put $b:=(f\circ\tilde{h})(\mathcal{B})$.
For every $t>0$, the inclusion $f(\tilde{h}(B_t))\subset U_t(b)$ gives
the component $V^{-1}_t$ of $f^{-1}(U_t(b))$ with
$\tilde{h}(B_t)\subset V^{-1}_t$. Then $V^{-1}_s\subset V^{-1}_t$ if $s<t$,
and $\bigcap_{t>0}V^{-1}_t$ is a singleton $\{a\}\subset f^{-1}(b)$.
Moreover, for every $r>0$, there is $s>0$ 
with $(\tilde{h}(B_s)\subset)V^{-1}_s\subset U_r(a)$, and hence
there exists the component $\tilde{A}_r$ of $\tilde{h}^{-1}(U_r(a))$
such that $B_s\subset\tilde{A}_r$.
Then $\tilde{\mathcal{A}}:=\{\tilde{A}_r\}\in\mathfrak{N}$.
Conversely, for every $r>0$, if $s>0$ is small enough,
then $((f\circ\tilde{h})(\tilde{A}_s)\subset)f(U_s(a))\subset U_r(b)$,
so that $(B_t\subset)\tilde{A}_s\subset B_r$ (for some $t>0$).
Hence $\tilde{\mathcal{A}}\sim\mathcal{B}$ and $\tilde{\mathcal{A}}\in\TS(\tilde{h})$.
Putting $\mathcal{A}:=\{\lambda^{-1}\tilde{A}_r\}\in\TS(\tilde{h}\circ\lambda)=\TS(h)$,
we have $\Lambda\mathcal{A}\sim\{\lambda(\lambda^{-1}\tilde{A}_r)\}=\tilde{\mathcal{A}}\sim\mathcal{B}$, so that $\Lambda\mathcal{A}=\mathcal{B}$ from (\ref{eq:equivalence}).\qed

\section{Proof of Theorem \ref{th:asym}}\label{sec:Schroder}

Let $f$ be a rational function on $\hbC$ of degree $d\ge 2$,
and $h$ a Schr\"oder map of $f$ as in \S \ref{sec:intro}.
Replacing $f^p$ by $f$ if necessary, 
we assume that period $p=1$, and
extend $\lambda:\bC\to\bC$ to $\Lambda:\TS(h)\to\TS(h)$ as
in Theorem \ref{th:extension}.
Fix $\mathcal{A}=\{A_r\}\in\TS(h)$, and 
put $a:=h(\mathcal{A})$ and $U_r:=U_r(a)$ $(r>0)$.
From (\ref{eq:extension}),
$\Lambda^{-k}\mathcal{A}\sim\{\lambda^{-k}A_r\}_{r>0}$ for every $k\in\bN$.

We introduce some notation.
The inclusion $f^k(h(\lambda^{-k}A_r))=h(A_r)\subset U_r$
gives the component $V^{-k}_r$ of $f^{-k}(U_r)$ such that 
$h(\lambda^{-k}A_r)\subset V^{-k}_r$. From the proof of 
the surjectivity of $\Lambda$ in \S \ref{sec:extension},
we have for every $k\in\bN$,
\begin{gather*}
 \bigcap_{r>0}V^{-k}_r=\{h(\Lambda^{-k}\mathcal{A})\}.
\end{gather*}

We first prove (\ref{eq:periodic}) and (\ref{eq:directasym}),
leaving (\ref{eq:unhyperbolic}) to the end.

\bold{Periodic $\mathcal{A}$:}
in the case that $\mathcal{A}$ is periodic under $\Lambda$,
without loss of generality,
we assume that $\Lambda^{-1}\mathcal{A}=\mathcal{A}$,
so that for every $k\in\bN$, $\mathcal{A}\sim\{\lambda^{-k}A_r\}_{r>0}$
and $\bigcap_{r>0}V^{-k}_r=\{a\}$.

When $f'(a)=0$, we immediately have $a\in\AT(f)$. 
Suppose that $f'(a)\neq 0$. If $r>0$ is small enough,
then for every $t\in(0,r]$, $f:V^{-1}_t\to U_t$ is univalent
and fixes $a$, and we denote its inverse by 
\begin{gather*}
 f^{-1}_t:U_t\to V^{-1}_t.
\end{gather*}
Since $\bigcap_{r>0}A_r=\emptyset$, 
diminishing $r>0$ if necessary, we may suppose that 
\begin{gather}
 A_r\subset\{|w|\ge 1\}.\label{eq:definite}
\end{gather}

If $a$ is a repelling or Siegel fixed point of $f$,
then there exists $t\in(0,r)$ such that 
for every $k\in\bN$, $(f^{-1}_r)^k$ is well defined on $U_t$ and 
\begin{gather*}
 h(\lambda^{-k}A_t)\subset V^{-k}_t=(f^{-1}_r)^k(U_t)\subset U_r,
\end{gather*}
so that $(A_s\subset)\lambda^{-k}A_t\subset A_r$ (for some $s>0$
since $\mathcal{A}\sim\{\lambda^{-k}A_r\}$),
which contradicts (\ref{eq:definite}) since always $|\lambda|^{-1}<1$.
Now we have proved (\ref{eq:periodic}). 
 
If $a\in\PB(f)\cup\CM(f)$, then for every $t\in (0,r)$,
Theorem \ref{th:hedgehog} yields a Siegel compactum $K_t$
associated to $(f^{-1}_{t/2},U_{t/2},a)$, 
so that $f^{-1}_t(K_t)=K_t\subset U_t$.
For every component $L$ of $h^{-1}(K_t)$ with $L\subset A_t$,
the inclusion $f(h(\lambda^{-1}L))=h(L)\subset K_t$ implies that
\begin{gather*}
 h(\lambda^{-1}L)\subset h(\lambda^{-1}A_t)\cap f^{-1}(K_t)\subset V^{-1}_t\cap f^{-1}(K_t)=f^{-1}_t(K_t)=K_t.
\end{gather*}
Let $\tilde{L}$ be the component of $h^{-1}(K_t)$
such that $\lambda^{-1}L\subset\tilde{L}$. Then 
$h(\lambda\tilde{L})=f(h(\tilde{L}))\subset f(K_t)=f(f^{-1}_t(K_t))=K_t$,
so that $(L\subset)\lambda\tilde{L}\subset L$.
From $\mathcal{A}=\Lambda^{-1}\mathcal{A}\sim\{\lambda^{-1}A_r\}$,
by decreasing $t\in(0,r)$ if necessary, we have
\begin{gather*}
 \tilde{L}=\lambda^{-1}L\subset\lambda^{-1}A_t\cap h^{-1}(K_t)\subset A_r\cap h^{-1}(U_t).\label{eq:induction}
\end{gather*}
If $\mathcal{A}$ were also direct, then diminishing $r>0$ if necessary,
we even have $A_r\cap h^{-1}(U_t)=A_t$ from Lemma \ref{th:directunique}.
Hence by induction,
for every $k\in\bN$, $\lambda^{-k}L$ must be a component of $h^{-1}(K_t)$ 
such that $\lambda^{-k}L\subset A_t(\subset A_r)$.
This contradicts (\ref{eq:definite}) as before.

Thus we have proved (\ref{eq:directasym})
since every direct transcendental singularity of $h$ is periodic under $\Lambda$
from Corollary \ref{th:postcrit} (\ref{item:periodic}).

\bold{Indirect $\mathcal{A}$:} 
we now assume that $\mathcal{A}$ is non-periodic and indirect, and show
(\ref{eq:unhyperbolic}) 
by eliminating any other possibility for $a=h(\mathcal{A}$). 

Suppose that $a\in F(f)$. For every $k\in\bN\cup\{0\}$, 
let $D^{-k}$ be the Fatou component of $f$ with
$h(\Lambda^{-k}\mathcal{A})\in D^{-k}$, so that $f(D^{-(k+1)})=D^{-k}$.

For every $k\in\bN\cup\{0\}$, $D^{-k}$ is cyclic:
otherwise, all $D^{-k}$ for $k$ large enough are not cyclic,
and are mutually disjoint. Then since $\#C(f)<\infty$, 
we have for all large $k$, $D^{-k}\cap C(f)=\emptyset$, i.e.,
$D^{-k}\cap(\bigcup_{\ell\in\bN}f^{\ell}(C(f)))=\emptyset$.
Hence by Corollary \ref{th:postcrit} (\ref{item:postcrit}),
all $\Lambda^{-k}\mathcal{A}$ for $k$ large enough
are not only distinct but direct, which contradicts (\ref{eq:DCA}).

From this fact and Corollary \ref{th:postcrit} (\ref{item:postcrit}),
we must have one of two alternatives: 
$a=h(\mathcal{A})\in\AT(f)$ (desirable) or
all $D^{-k}$ are rotation domains of $f$.
However, the second situation cannot occur: fix $r>0$ with $U_r\subset D^0$. 
Then for every $k\in\bN$, 
$h(\Lambda^{-k}\mathcal{A})\in V^{-k}_r\subset D^{-k}$, so that
\begin{gather*}
 h(\lambda^{-k}A_r)\subset D^{-k}.
\end{gather*}
If all $D^{-k}$ were rotation domains, then 
$f^k:D^{-k}\to D^0$ is univalent for every $k\in\bN$.
Since $h'(0)\neq 0$, $h|\{|w|<t\}$ is univalent for $t>0$ small enough,
and then by repeated use of (\ref{eq:Schroder}),
\begin{multline*}
 (h:\{w\in A_r;|w|<|\lambda|^kt\}\to D_0)\\
=(f^k:D^{-k}\to D^0)\circ(h:\{w\in\lambda^{-k}A_r;|w|<t\}\to D^{-k})
\circ\lambda^{-k}
\end{multline*}
is univalent for all $k\in\bN$. Hence $A_r\cap C(h)=\emptyset$,
which contradicts Theorem \ref{th:indirect} since $\mathcal{A}$ is indirect.

Finally, suppose that $a\in J(f)\setminus(\PB(f)\cup M(f))$.
We apply Theorem \ref{th:recurrent} to $U=U_r(a)$ for $r>0$ small enough.
Fix $t>0$ such that $h|\{|w|<t\}$ is univalent, and for this $t>0$,
put $\phi_t:=(h|\{|w|<t\})^{-1}:h(\{|w|<t\})\to\{|w|<t\}$.
Also fix $s>0$ such that $U_{2s}(z_0)\Subset h(\{|w|<t\})$
($z_0=h(0)$).
For all large $k\in\bN$, (\ref{eq:schrink}) shows that 
$\diam V^{-k}_r<s$, and since $0\in\phi_t(U_s(z_0))$,
\begin{gather}
 \lambda^{-k}A_r\cap\phi_t(U_s(z_0))\neq\emptyset.\label{eq:accumulation}
\end{gather} 
Hence from $h(\lambda^{-k}A_r)\subset V^{-k}_r$, we have
$V^{-k}_r\cap U_s(z_0)\neq\emptyset$, and since $\diam V^{-k}_r<s$,
$V^{-k}_r\subset U_{2s}(z_0)$.
We recall again (\ref{eq:accumulation}) and deduce that
$\lambda^{-k}A_r\subset\phi_t(U_{2s}(z_0))$. This cannot be true
since $A_r$ is unbounded. \qed

\section{Proof of Theorem \ref{th:attpara}}\label{eq:asympath}
Replacing $f$ by an appropriate iterate if necessary,
we assume that $N=1$, so that $\lambda W=W$, $f(D)=D$ and $f(a)=a$.

Fix $w_0\in W$, consider an arc $\gamma:[0,1]\to W$
with $\gamma(0)=w_0$ and $\gamma(1)=\lambda w_0$, and
extend the domain of $\gamma$ to $(-\infty,\infty)$ via the functional equation
(\ref{eq:translation}) (with $N=1$). Then for every $k\in\bN$ and every $s\in[0,1]$,
$\gamma(k+s)=\lambda^k(\gamma(s))\in\lambda^k(\gamma([0,1]))$
and $h(\gamma(k+s))\in f^k(h(\gamma([0,1])))$. 
Since $\gamma([0,1])\subset(W\subset)\bC^*$ and $h(\gamma([0,1]))\subset D$,
we have $\gamma(t)\to\infty$ and $h(\gamma(t))\to a$ as $t\to\infty$.
Thus $\gamma$ is as described in Theorem \ref{th:attpara},
to which $\mathcal{A}^{\gamma}\in\TS(h)$ may be associated.

From now on, suppose that $a\in\AT(f)$.

If $\gamma(0)=w_0\in W\cap h^{-1}(a)$, then by $f(a)=a$ and
(\ref{eq:Schroder}), $\gamma(k)=\lambda^kw_0\in h^{-1}(a)$ for every $k\in\bN$.
Hence, since $\gamma(k)\to\mathcal{A}^{\gamma}$ as $k\to\infty$,
$\mathcal{A}^{\gamma}$ is indirect.

Next, suppose that $f^{-1}(a)\cap D\neq\{a\}$. 
First, putting the {\itshape backward orbits}
$\mathcal{BO}:=\bigcup_{n\in\bN}(f:D\to D)^{-n}(a)$ of $a$ in $D$, 
we have $\#\mathcal{BO}=\infty$: otherwise,
since $\mathcal{BO}$ is $(f:D\to D)^{-1}$-invariant,
every $a'\in\mathcal{BO}$ must be periodic, and then $a'$ must equal $a$,
so that $f^{-1}(a)\cap D=\{a\}$, which is a contradiction.
Second, we also have $\#(D\setminus h(W))<\infty$:
for each $b\in D\setminus h(W)$, there exists $\mathcal{B}=\{B_r\}\in\mathfrak{N}$
such that for all small $r>0$, $B_r$ is a component of $U_r(b)$ with $B_r\subset W$.
We claim that $\bigcap_{r>0}B_r=\emptyset$:
otherwise, $\bigcap_{r>0}B_r$ must be a singleton
in $h^{-1}(b)$, and hence $h(W)\supset h(\bigcap_{r>0}B_r)=\{b\}$,
which is a contradiction. Hence $\mathcal{B}\in\TS(h)$ with $h(\mathcal{B})=b$,
and $\mathcal{B}$ must be direct since $b\not\in h(W)$. 
This with (\ref{eq:DCA}) 
yields $\#(D\setminus h(W))\le\#(\text{direct singularities of }h)<\infty$.
Consequently, $\#(\mathcal{BO}\cap h(W))=\infty$,
which provides $w_1\in W$ such that $f^n(h(w_1))=a$ for some $n\in\bN$.
Thus $h(\lambda^nw_1)=a$ from (\ref{eq:Schroder}), so that
$\lambda^nw_1\in \lambda^nW\cap h^{-1}(a)=W\cap h^{-1}(a)$. 
Hence $W\cap h^{-1}(a)\neq\emptyset$. 

Finally, suppose that some $\mathcal{A}=\{A_r\}\in\TS(h)$ 
with $h(\mathcal{A})=a$ is indirect 
and that $A_r\subset W$ for some $r>0$.
Then by Theorem \ref{th:indirect},
there is $c\in A_r\cap C(h)\setminus h^{-1}(a)(\subset W)$.
Fix $t>0$ such that $h|\{|w|<t\}$ is univalent (using $h'(0)\neq 0$),
and fix $\ell\in\bN$ such that $\lambda^{-\ell}c\in\{|w|<t\}$ 
(using $|\lambda|>1$).
By (\ref{eq:Schroder}) and the chain rule,
\begin{gather*}
 0=h'(c)=(f^{\ell}\circ h\circ\lambda^{-\ell})'(c)
  =\prod_{i=0}^{\ell-1}f'(f^i(h(\lambda^{-\ell}c)))\cdot h'(\lambda^{-\ell}c)\cdot\lambda^{-\ell},
\end{gather*} 
which implies that $f^i(h(\lambda^{-\ell}c))\in C(f)$ for some $i\in\{0,\ldots,\ell-1\}$. 
For this $i$, from (\ref{eq:Schroder}) and $f(a)=a\neq h(c)$, we also have
\begin{gather*}
  a\neq f^i(h(\lambda^{-\ell}c))
 =h(\lambda^{-(\ell-i)}c)\in h(\lambda^{-(\ell-i)}W)=h(W)\subset D.
\end{gather*} 
Hence $D\cap C(f)\neq\{a\}$.
\qed

\section{Proofs of Corollaries \ref{th:PLY} and \ref{th:hypdense}}\label{sec:example}

\begin{proof}[Proof of Corollary $\ref{th:PLY}$]

For a transcendental entire function $g$ and $\mathcal{A}=\{A_r\}\in\TS(g)$,
we denote by $A_r^*(0)$ the component of $\bC\setminus A_r$ containing $0$
(if exists), and say that $\mathcal{A}$ is {\itshape non-annular} if
$A_r^*(0)$ is unbounded for all small $r>0$.
Note that if $\mathcal{A}^j=\{A_r^j\}\in\TS(g)$
with $g(\mathcal{A}^j)=\infty$ ($j=1,\ldots,q'$) are mutually distinct,
then $\{\mathcal{A}^j\}_j$ must be {\itshape totally separated} in that 
\begin{gather}
 \bigcup_{i\neq j}A_r^i\subset (A_r^j)^*(0)\ (j=1,\ldots,q')\label{eq:separated} 
\end{gather}
for all small $r>0$. (\ref{eq:separated}) implies that when $q'\geq 2$,
all $\mathcal{A}^j$ ($j=1,\ldots,q'$) are non-annular.

To prove (\ref{eq:general}), we need a spiral version of Ahlfors's theorem
(see our discussion of (\ref{eq:general})),
which we formulate here as Theorem \ref{th:spiral};
the argument which seems useful to us is from Jenkins \cite[\S 3]{Jenkins87}
and Hayman \cite[Theorem 8.21]{Hayman89}.

\begin{theorem}\label{th:spiral}
 Let $g$ be an entire function of finite order $\rho$,
 consider mutually distinct $\mathcal{A}^j\in\TS(h)$ with $g(\mathcal{A}^j)=\infty$
 $(j=1,\ldots,q')$, and suppose that $\mathcal{A}^1$ is non-annular when $q'=1$.
 Then for every asymptotic arc $\gamma$ of $h$ to which $\mathcal{A}^1$ is associated,
 \begin{gather}
  q'\cdot\left(1+\limsup_{t\to\infty}\left(\frac{\arg\gamma(t)}{\log|\gamma(t)|}\right)^2\right)\le 2\rho.\label{eq:spiralinfinite}
 \end{gather}
\end{theorem}

\begin{remark}
 The assumption that $\mathcal{A}^j$ is non-annular
 is required because we work with the image of $A_r^j$
 under a branch of logarithm in the proof. 
 An entire function of order $<1/2$ shows that such a condition is essential.
\end{remark} 

We have already hinted at the proof of Corollary \ref{th:PLY} in \S \ref{sec:intro}.

By Eremenko and Levin, 
each component $W$ of $h^{-1}(D_{\infty})$ is periodic under $\lambda$. 
Thus $W\subset h^{-1}(F(f))$, $0\in\overline{W}\cap h^{-1}(J(f))$
(since $|\lambda|>1$) and
\begin{gather}
 0\in\partial W.\label{eq:boundary} 
\end{gather}

For a given $W$, Theorem \ref{th:attpara} yields an asymptotic arc 
$\gamma_W:(-\infty,\infty)\to W$ with $\lim_{t\to\infty}h(\gamma_W(t))=\infty$,
to which $\mathcal{A}^W=\{A_r^W\}\in\TS(h)$ may be associated.
We check that $\{\mathcal{A}^W\}_W$ is totally separated, and
that (\ref{eq:ELcondition}) implies that all $\mathcal{A}^W$ are non-annular.
Then Theorem \ref{th:spiral} may be applied to $\{\mathcal{A}^W\}_W$,
and all $\gamma_W$ satisfy (\ref{eq:general}) (for $q'=q_{\infty}$).

For all small $r>0$, $U_r(\infty)\subset D_{\infty}$, so that
\begin{gather}
 A^W_r\subset W.\label{eq:contained}
\end{gather}
Hence (\ref{eq:boundary}) and (\ref{eq:contained}) show that
$\{\mathcal{A}^W\}_W$ is totally separated.

Let $\tilde{J_0}$ be the component of $h^{-1}(J(f))$ containing $0$.
Since $W\subset h^{-1}(F(f))$, $\tilde{J_0}\cap W=\emptyset$.
Hence from (\ref{eq:boundary}) and (\ref{eq:contained}), we have 
\begin{gather}
 \tilde{J_0}\subset(A^W_r)^*(0),\label{eq:Julia}
\end{gather}
using the notation introduced at the beginning of this section. 
By the $f$-invariance of $J(f)$ and (\ref{eq:Schroder}), we also have
\begin{gather*}
 \lambda\tilde{J_0}\subset\tilde{J_0}. 
\end{gather*}

From now on, we assume (\ref{eq:ELcondition}), which is equivalent to
\begin{gather*}
 \tilde{J_0}\neq\{0\} 
\end{gather*}
since $h$ is analytic near $0$ (and $h(0)=z_0$).
Then $\tilde{J_0}$ is unbounded (since $|\lambda|>1$) and 
$\mathcal{A}^W$ is non-annular (by (\ref{eq:Julia})). 

The fact that $\tilde{J_0}$ is unbounded (when (\ref{eq:ELcondition}) holds)
also implies that $W$ can contain no closed curve winding around $0$,
so that there is a single-valued branch $\arg_W(\cdot)$ of $\arg(\cdot)$ on $W$.

Using Theorem \ref{th:attpara}, we may assume in addition that 
$\gamma_W$ satisfies (\ref{eq:translation}), and hence (\ref{eq:orbits}).
For each $W$ and each $R>0$, let $t_W(R)$ be the least $t>0$
such that $|\gamma_W(t)|=R$, and put $P_W(R):=\gamma_W(t_W(R))$.
The set $\{P_W(R)\}_W$ has a natural cyclic order on the
(counterclockwise-oriented) circle $\{|w|=R\}$,
which induces a cyclic order of components $W$ of $h^{-1}(D_{\infty})$
such that $W$ is the $j$-th component of $h^{-1}(D_{\infty})$ if and only if
$P_W(R)$ is the $j$-th point in $\{P_W(R)\}_W$. Let us denote by $W_j$
the $j$-th component of $h^{-1}(D_{\infty})$ ($j=0,\ldots,q_{\infty}-1$).
Replacing $\gamma_{\lambda W}$ if necessary, 
we may further assume that for every $W$, $\gamma_{\lambda W}=\lambda\cdot\gamma_W$,
so that $P_{\lambda W}(|\lambda|R)=\lambda\cdot P_W(R)$.
This yields the unique $p_{\infty}=p_{\infty}(h)\in\{0,\ldots,q_{\infty}-1\}$
such that for each $j=0,\ldots,q_{\infty}-1$,
\begin{gather}
 \lambda W_j=W_{j+p_{\infty}(\text{mod }q_{\infty})}.\label{eq:cyclic}
\end{gather}

Hence, recalling that $q_W=\min\{N\in\bN;\lambda^N W=W\}$,
we observe that $q:=q_W$ is independent of $W$ and that 
$q_{\infty}/q=:m_{\infty}=m_{\infty}(h)\in\bN$ is 
the number of cycles of components of $h^{-1}(D_{\infty})$ under $\lambda$.
We also recall $p_W$ from (\ref{eq:argument}), 
and observe that $p_W\cdot q_{\infty}=q_W\cdot p_{\infty}$. 
Hence $p_W=p_{\infty}/m_{\infty}=:p$ is also independent of $W$.

Now the proof of Corollary \ref{th:PLY} is completed by
the spiral inequality (\ref{eq:general}) for $\gamma_W$ together with the calculation
\begin{gather*}
\frac{\arg_W(\gamma_W(k))}{\log|\gamma_W(k)|}
=\frac{k(q\arg\lambda-2\pi p)+\arg_W(\gamma_W(0))}{\log(|\lambda|^{k\cdot q}|\gamma_W(0)|)}
\end{gather*}
from (\ref{eq:orbits}) and (\ref{eq:argument}) for any $W$.
\end{proof}

\begin{remark}
The fact that $\infty\in D_{\infty}$ also implies that
every asymptotic arc $\gamma:(-\infty,\infty)\to\bC$
with $\lim_{t\to\infty}h(\gamma(t))=\infty$ is contained in some component
$W$ of $h^{-1}(D_{\infty})$. Hence to show (\ref{eq:general}) for 
this $\gamma$, we can always take $\gamma$ as $\gamma_W$ from the above proof.
\end{remark}

\begin{proof}[Proof of Corollary $\ref{th:hypdense}$]
 The following, which uses standard ideas but appears to be new,
 may have independent interest. The hypothesis of finite order is essential.
 For completeness, we include the proof in \S \ref{sec:kysymus}.

 \begin{mainth}\label{th:Kysymus}
 Let $g$ be an entire function of finite order $\rho$.
 If $g$ does not cover $a\in\bC$ completely
 $($in the same sense as discussed before Corollary $\ref{th:hypdense})$,
 then $a\in g(\TS(g))$.
 \end{mainth}

 By Corollary \ref{th:postcrit} (\ref{item:entire}), 
 every Schr\"oder map $h$ of $f_c$ must be entire,
 and hence by Corollary \ref{th:postcrit} (\ref{item:periodic}),
 every $\mathcal{A}\in\TS(h)$ is periodic.
 
 Let $c\in\Int\mathcal{C}$ (recall $\mathcal{C}$ from the end of \S 1).
 Then $f_c$ has no indifferent periodic point
 (cf.\ \cite[Theorem 4.8]{McMullen:renorm}), and hence
 from (\ref{eq:periodic}) of Theorem \ref{th:asym},
 $h(\TS(h))\subset\AT(f_c)$ for every Schr\"oder map $h$ of $f_c$.

 If a Schr\"oder map $h$ of $f_c$ does not cover $a\in\bC$ completely,
 then $a\in h(\TS(h))$ by Theorem \ref{th:Kysymus},
 so that as we just observed, 
 we have even $a\in\AT(f_c)$, which implies that $c\in\mathcal{H}$.
 Conversely, if $c\in\mathcal{H}$, then there is $a\in\AT(f_c)\cap\bC$,
 and by Corollary \ref{th:attparaconverse},
 there is a Schr\"oder map $h$ of $f_c$ such that $a\in h(\TS(h))$.
 Clearly, $h$ cannot cover $a$ completely.
\end{proof}

\section{Proof of Theorem \ref{th:Kysymus}}\label{sec:kysymus}
The hypotheses guarantee that $g$ is transcendental.
A consequence of this (with the identity theorem) is that
for every $r>0$, the cardinality of a subset of $S(r):=\{|z|=r\}$
on which either $|g|$ or $\arg g$ (mod $2\pi$) is constant must be finite.

We may suppose that $a=0$, and consider a sequence $(r_n)\subset\bR_{>0}$
with $r_n\searrow 0$ as $n\to\infty$ and a sequence of unbounded components
$\Delta_n$ of preimages of $\bD_n=\{|w|<r_n\}$ under $g$.
It is enough to show that if $p>2\rho$, 
then $\Delta_n$ ($n=1,\ldots, p$) cannot be mutually disjoint.

Suppose that $\Delta_n$ ($n=1,\ldots, p$) were mutually disjoint.
Increasing each $r_n$ slightly if necessary,
we can assume that no critical value of $g$ lies on $\bigcup_{n=1}^p\{|w|=r_n\}$.
From the observation above, the boundary $\partial\Delta_n$ ($n=1,\ldots,p$)
consists of finitely many (analytic) Jordan arcs whose endpoints are both $z=\infty$.

The next lemma contains the main idea.

\begin{lemma}\label{th:Lindelof}
Let $F$ be a simply-connected unbounded component of 
$\{|z|>r\}\setminus(\Delta_m\cup\Delta_n)$, where $m\neq n$,
whose boundary consists of a subarc of $S(r)$ for some $r>0$
and unbounded Jordan subarcs $\gamma$ of $\partial\Delta_m$ and
$\gamma'$ of $\partial\Delta_n$, each of which has an endpoint on $S(r)$.
Then $g$ is unbounded in $F$.
\end{lemma}

\begin{proof}
Suppose on the contrary that $g$ is bounded in $F$.
Then by Lindel\"of's theorem \cite[p.\ 75]{Nevan70},
if $g(z)$ converges as $z\to\infty$
along two asymptotic arcs in $\overline{F}$ and 
the limits are both in $\bC$, then they must coincide.

At once from the Cauchy-Riemann equations,
(any fixed branch of) $\arg g$ (not mod $2\pi$) varies monotonically
along each of the arcs $\gamma,\gamma'\subset\overline{F}$.
If the variation of $\arg g$ were bounded on each of $\gamma$ and $\gamma'$,
then $g(z)$ converges as $z\to\infty$ along each of them,
and by Lindel\"of's theorem, the limits must be same, 
which cannot be since $r_m\neq r_n$.

Thus, with no loss of generality, we may choose any fixed branch
of $\arg g$ on $\gamma$, and it will be unbounded on $\gamma$.
On the other hand, there are (in fact uncountably many)
distinct $\varphi_1,\varphi_2\in [0, 2\pi)$ such that
no critical value $w$ satisfies $\arg w=\varphi_1$ or $\varphi_2$ (mod $2\pi$).

For each $j=1,2$, we may choose infinitely many mutually distinct 
$z_{jk}\in\gamma$ $(k\in\bN)$ with $\arg g(z_{jk})=\varphi_j$ (mod $2\pi$).
For each $z_{jk}$, there exists the (maximal) lift $\Gamma_{jk}\subset\bC_z$ by $g$
of the radial ray $(r_m,\infty)\ni t\mapsto te^{i\varphi_j}\in\bC_w$
with an endpoint $z_{jk}$. 
Since the ray $\{w\in\bC;\arg w=\varphi_j\}$ is free of critical values, 
$\Gamma_{jk}\cap \Gamma_{jk'}=\emptyset$ if $k\neq k'$. 
In addition, as we noted in the beginning of this section,
$\Gamma_{jk}$ can intersect $S(r)$ for at most finitely many $k$.
On the other hand, for every $k$, $\Gamma_{jk}$ cannot intersect 
$\gamma'\cup\gamma$ since $|(g|F)|>r_m$ near $\gamma$ and $|(g|F)|>r_n$ near $\gamma'$.
Consequently, for all large $k$, the maximal lift $\Gamma_{jk}$
is a Jordan arc in $F$ tending to $\infty$. Then
since $g$ is bounded in $F$, there must exist $t=t_{jk}\in(r_m,\infty)$ such that
$g(z)\to te^{i\varphi_j}$ as $\Gamma_{jk}\ni z\to\infty$.

However, by Lindel\"of's theorem,
all the limits coincide for $j=1,2$ and all large $k$,
so that $\varphi_1=\varphi_2$, which is a contradiction.
\end{proof}

Now we complete the proof of the theorem.
Once we fix $r>0$ large enough,
we obtain at least $p$ of mutually disjoint domains $F$
satisfying the hypotheses of Lemma \ref{th:Lindelof}.
However,

\begin{lemma}\label{th:setF}
 If $p>2\rho$, then $g$ must be bounded on at least one of these $F$.
\end{lemma}

\begin{proof}
We need a standard estimate of {\itshape harmonic measure} $\omega$
(cf.\ \cite[XI. \S 4]{Nevan70}).
In the following,
$C_1,C_2(>2\log 2),C_3,C_4$ denote some positive constants independent of $R>0$:

Let $F\subset\bC$ be an arbitrary domain intersecting $S(t)$ 
so that $S(t)\setminus F\neq \emptyset$ for every $t\ge r$,
and let $\theta_F(t)$ the angular measure of $F\cap S(t)$. Then
 \begin{gather*}
  \omega(\cdot,F,F\cap S(R))<C_1\exp\left(-\pi\int_{2r}^{R/2}\frac{\rd t}{t\theta_F(t)}\right)\text{ on } F\cap A(2r,R/2)\label{eq:tsuji}
 \end{gather*}
as soon as $R/2>2r$, where $A(r,R):=\{r<|z|<R\}$. Let $F_1,\ldots, F_\ell\subset\bC$
($\ell\in\bN$) be arbitrary mutually disjoint domains intersecting $S(t)$
for every $t\ge r$, with angular measure $\theta_j(t):=\theta_{F_j}(t)$ as above.
Note that $\sum_{j=1}^{\ell}\theta_j(t)\le 2\pi$ for all such $t$.
Then by an argument involving the Cauchy-Schwarz inequality,
we find some $j\in\{1,\ldots,\ell\}$
and a sequence $(R_k)\subset\bR$ tending to $\infty$ as $k\to\infty$
such that for every $R=R_k$,
\begin{gather*}
 \pi\int_{2r}^{R/2}\frac{\rd t}{t\theta_j(t)}\ge
 \frac{\ell}{2}\left(\log\frac{R}{r}-C_2\right).\label{eq:harm}
\end{gather*}

Applying these estimates to $p$ of our domains $F$,
we conclude that for one of these $F$, 
there is a sequence $(R_k)\subset\bR$ with $R_k\nearrow\infty$ as $k\to\infty$
such that for each $R=R_k$,
\begin{gather*}
 \omega(\cdot,F,F\cap S(R))<C_3(R/r)^{-p/2}\text{ on }F\cap A(2r,R/2),
\end{gather*}
while (since $(r_m)$ decreases,) 
we have $\log |g|<\log r_1+(\max_{S(r)}\log|g|)<C_4$ on $\partial F$.
Since $g$ is entire, we also have
\begin{gather*}
 \rho=\limsup_{r\to\infty}\frac{\log\max_{|z|=r}\log^+|g(z)|}{\log r},\label{eq:orderentire}
\end{gather*}
from which, for all small $\epsilon\in(0,(p-2\rho)/2)$,
we have for all large $R$, $\log|g|<R^{\rho+\epsilon}$ on $S(R)$.
Thus by the two-constants theorem, for all large $R=R_k$,
\begin{gather*}
 \log|g|\le C_4+R^{\rho+\epsilon}\cdot C_3(R/r)^{-p/2}<2C_4\text{ on }F\cap A(2r,R/2).
\end{gather*}
Hence $g$ is bounded on $F$.
\end{proof}

\begin{ac}
 The authors thank the Department of Mathematics and Statistics,
 University of Helsinki for providing a stimulating atmosphere
 and making this collaboration possible. The first author
 was on sabbatical from Purdue University, and
 the second author was supported by the exchange program of scientists
 between the Academy of Finland and the Japan Society for Promotion of Science.
 The second author also thanks the Purdue Mathematics Department for
 supporting his visit to Purdue, where this was completed.
\end{ac}

\begin{supplement}
 Since this issue is dedicated to Walter Hayman,
 I want to indicate the significant influence he has had in my career. 
 I am sure that my experiences are not unusual, but usually there are no opportunities
 to make such comments.

 His book {\itshape Meromorphic Functions} went through two editions, and even today
 may be the most efficient introduction to classical Nevanlinna theory in one
 complex variable (it was also the first text in English!).  The first two chapters develop the fundamental
 theorems established by Rolf Nevanlinna in the 1920s, but every other chapter
 was centered on a topic which quickly led the reader to open questions, whose
 resolution has been a major activity for 30-plus years (Chapter 4 was
 my focus).   His series of  problem compilations,
 an activity beginning a few years before MF, efficiently covered a large
 variety of fields related to one complex variable (including some
 higher-dimensional questions).  In those pre-internet days, assembling the
 material from colleagues all over the world required special effort.
 Until recent times (circa 1990) international contact was difficult in
 many countries, and his visits, letters and collections played an important
 role in supporting colleagues and students from these countries.

 In his dealings with colleagues, he was always encouraging, and raising
 interesting questions for study.
 We met in 1967, Montreal, where I asked him a question about minimum
 modulus. Several weeks later, I received his example, which was the
 kernel of several of my later works (some joint with Dan Shea).
 At conferences he always had time to talk to participants, in a way
 that was always enthusiastic, supportive; he treated everyone with
 the same courtesy and spirit.
\end{supplement}

\def\cprime{$'$} \def\polhk#1{\setbox0=\hbox{#1}{\ooalign{\hidewidth
  \lower1.5ex\hbox{`}\hidewidth\crcr\unhbox0}}}

\end{document}